\documentclass[11pt,a4paper]{article}
\usepackage{amssymb} 
\usepackage{amsmath} 
\usepackage{amsthm} 
\usepackage{hyperref}

\newtheorem{theorem}{Theorem}
\newtheorem{proposition}[theorem]{Proposition}
\newtheorem{lemma}[theorem]{Lemma}
\newtheorem{corollary}[theorem]{Corollary}
\newtheorem{conjecture}[theorem]{Conjecture}
\newtheorem{observation}[theorem]{Observation}

\newtheorem*{chernoff}{The Chernoff Bound}
\newtheorem*{slll}{The Lov\'asz Local Lemma}
\newtheorem*{glll}{The General Local Lemma}

\renewcommand{\Pr}{\,\mathbb{P}}

\DeclareMathOperator{\Bin}{Bin}
\DeclareMathOperator{\ch}{ch}
\DeclareMathOperator{\sep}{sep}
\DeclareMathOperator{\con}{\nleftrightarrow}
\newcommand{\chicon}{\chi_{\con}}

\renewcommand{\labelenumi}{\theenumi}

\title{Single-conflict colouring\footnote{This research was partly supported by a Van Gogh grant, reference 35513NM.}}

\author{
Zden\v{e}k Dvo\v{r}\'ak
\thanks{Computer Science Institute (CSI) of Charles University,
           Malostransk{\'e} n{\'a}m{\v e}st{\'\i} 25, 118 00 Prague, 
           Czech Republic. E-mail: \protect\href{mailto:rakdver@iuuk.mff.cuni.cz}{\protect\nolinkurl{rakdver@iuuk.mff.cuni.cz}}.
           Supported by (FP7/2007-2013)/ERC Consolidator grant LBCAD no. 616787.}
\and
Louis Esperet
\thanks{Universit\'e Grenoble Alpes, CNRS, G-SCOP, Grenoble, France.
Email: \protect\href{mailto:louis.esperet@grenoble-inp.fr}{\protect\nolinkurl{louis.esperet@grenoble-inp.fr}}.
This author is partially supported by ANR Projects STINT
  (\textsc{anr-13-bs02-0007}) and GATO (\textsc{anr-16-ce40-0009-01}), and LabEx PERSYVAL-Lab
  (\textsc{anr-11-labx-0025}).
}
\and
Ross J. Kang
\thanks{Department of Mathematics, Radboud University Nijmegen, Netherlands. 
Email: \protect\href{mailto:ross.kang@gmail.com}{\protect\nolinkurl{ross.kang@gmail.com}}. Supported by a Vidi grant (639.032.614) of the Netherlands Organisation for Scientific Research (NWO). This author is grateful to the hospitality of Yokohama National University, where part of this research originated.}
\and
Kenta Ozeki
\thanks{Faculty of Environment and Information Sciences, Yokohama National University, Japan. 
Email: \protect\href{mailto:ozeki-kenta-xr@ynu.ac.jp}{\protect\nolinkurl{ozeki-kenta-xr@ynu.ac.jp}}.}
}

\begin{document}

\maketitle

\begin{abstract}
Given a multigraph, suppose that each vertex is given a local assignment of $k$ colours to its incident edges. We are interested in whether there is a choice of one local colour per vertex such that no edge has both of its local colours chosen. The least $k$ for which this is always possible given any set of local assignments we call the {\em single-conflict chromatic number} of the graph. This parameter is closely related to separation choosability and adaptable choosability. We show that single-conflict chromatic number of simple graphs embeddable on a surface of Euler genus $g$ is $O(g^{1/4}\log g)$ as $g\to\infty$. This is sharp up to the logarithmic factor.

\end{abstract}

\section{Introduction}\label{sec:intro}

Dvo\v{r}\'{a}k and Postle~\cite{DvPo17} and Fraigniaud, Heinrich and Kosowski~\cite{FHK16} independently defined the {\em conflict $k$-colouring} problem as follows. Given a (simple) graph $G=(V,E)$, 
each edge $uv\in E$ is assigned a list $K(u,v)$ of ordered pairs ---called {\em conflicts}--- of colours from $[k]=\{1,\dots,k\}$.
The question is whether $G$ admits a colouring $c:V\to[k]$ of the vertices so that no edge is in a conflict, i.e.~there is no edge $uv \in E$ and conflict $(\kappa_u,\kappa_v)\in K(u,v)$ such that $c(u)=\kappa_u$ and $c(v)=\kappa_v$.
The authors in~\cite{DvPo17} and~\cite{FHK16} also imposed further natural restrictions based on contrasting goals and perspectives, but here instead we only prescribe the maximum number $\mu$ of conflicts per edge.

In fact, this is equivalent to the ``least nontrivially conflicting'' version of the problem, with exactly one conflict per edge, provided we pass to a {\em multigraph} of maximum edge multiplicity $\mu$. (This motivates our first use of the letter $\mu$.) 
Let us be more precise. Let $G=(V,E)$ be a multigraph. For any positive integer $k$,  a {\em local $k$-partition} of $G$ is a collection $\{\kappa_v\}_{v\in V}$ of maps of the form $\kappa_v: E(v) \to [k]$, where $E(v)$ denotes the set of edges incident to $v$. So each $\kappa_v$ is a partition of $E(v)$ into $k$ parts, and for each $e\in E(v)$ the colour $\kappa_v(e)$ can be thought of as the {\em local} colour\footnote{By relabelling, we alternatively may define the $\kappa_v$ as maps from $E(v)$ to $\mathbb{N}$ each image set of which contains at most $k$ elements, so not necessarily the same image for every $v$.} of $v$ associated to $e$. 
For each edge $uv$ in the underlying simple graph we want the set $K(u,v)$ of conflicts to be composed of the pairs $(\kappa_u(e),\kappa_v(e))$, for every $e\in E$ with endpoints $u$ and $v$.
Thus, given a local $k$-partition $\{\kappa_v\}$, we say $G$ is {\em conflict $\{\kappa_v\}$-colourable} if there is some colouring $c: V\to [k]$ of the vertices so that no edge $e=uv\in E$ has $c(u)=\kappa_u(e)$ and $c(v)=\kappa_v(e)$.
The {\em single-conflict chromatic number} $\chicon(G)$ of $G$ is the smallest $k$ such that $G$ is conflict $\{\kappa_v\}$-colourable for any local $k$-partition $\{\kappa_v\}_{v\in V}$.

As we discuss in Section~\ref{sec:definitions}, single-conflict chromatic number considerably strengthens upon two notable list colouring parameters, separation choosability (cf.~\cite{KTV98a}) and adaptable choosability (cf.~\cite{KoZh08}), and so its study could potentially yield new insights into these two parameters.

Before continuing, we give two easy but instructive examples. First, for a square integer $\mu$, consider two vertices with $\mu$ edges between them. Take the local $\sqrt{\mu}$-partition which lists all $\mu$ possible pairwise conflicts between the two vertices. So this is a $\mu$-edge planar multigraph with maximum degree and multiplicity both $\mu$ that has single-conflict chromatic number greater than $\sqrt{\mu}$.
Second, for positive integers $a$ and $\mu$, consider $(a\mu)^a$ vertices, written $v_{i_1,\dots,i_a}$, $i_1,\dots i_a\in[a\mu]$, each joined by $\mu$ edges to each of $a$ vertices $u_1,\dots,u_a$. Take the local $(a\mu)$-partition where for each $j\in[a]$ and $i_1,\dots,i_a\in [a\mu]$  the edges between $u_j$ and $v_{i_1,\dots,i_a}$ include all conflicts having $i_j$ as a local colour for $u_j$ and some integer in $[j\mu]\setminus [(j-1)\mu]$ as a local colour for $v_{i_1,\dots,i_a}$. For any $i_1,\dots i_a\in[a\mu]$, if $u_j$ is coloured $i_j$ for all $j\in[a]$, there remains no possibility for colouring $v_{i_1,\dots,i_a}$. 
Thus the multigraph formed from the complete bipartite graph $K_{a,(a\mu)^a}$ by multiplying each edge by $\mu$ has single-conflict chromatic number greater than $a\mu$.
(Note that by substituting $(a\mu)^a$ copies of the first example, one can boost this to $a\mu+\Omega(\sqrt{\mu})$.)
With $a=2$ this yields a planar multigraph of maximum multiplicity $\mu$ that has single-conflict chromatic number greater than $2\mu$. 

Besides introducing single-conflict chromatic number and setting down some of its basic behaviour, our main task in this paper is to treat it in a classic setting for chromatic graph theory.
We prove the following.

\begin{theorem}\label{thm:surfaces}
For some constant $C_1>0$, if $G$ is a multigraph of maximum multiplicity $\mu\ge 1$ that is embeddable on a surface of Euler genus $g$, then $\chicon(G) \le \max\{C_1 \sqrt{\mu}(g+1)^{1/4} \log (\mu^2(g+2)),8\mu\}$.
\end{theorem}

\noindent
Note that the $8\mu$ term cannot be lowered to $2\mu$ due to the second example above.
We will see below that the other term is sharp up to at most a polylogarithmic factor.

Allow us to reiterate the $\mu=1$ case, which may be interpreted as an analogue of Heawood's classic formula for the chromatic number~\cite{Hea90}.
\begin{corollary}
For some constant $C>0$, if $G$ is a simple graph that is embeddable on a surface of Euler genus $g$, then $\chicon(G) \le C (g+1)^{1/4} \log (g+2)$.
\end{corollary}

\noindent
The $\mu=\Theta(\sqrt{g})$ case is also of special interest, in which case Theorem~\ref{thm:surfaces} implies a bound of the form $\chicon(G) = O(\sqrt{g}\log g)$. 
Since the number of colours available per vertex is close to the maximum multiplicity and both are around $\sqrt{g}$, this result is evocative of Heawood's bound itself.
Naturally one could aspire towards an elimination of the logarithmic factor.
\begin{conjecture}\label{conj:heawood}
There exist $C,C'>0$ such that, for any simple graph $G$ that is embeddable on a surface of Euler genus $g$, if every edge is assigned at most $C k$ conflicts from $[k]^2$, then $G$ is conflict $k$-colourable, provided $k\ge C' \sqrt{g}$.
\end{conjecture}
\noindent
Note that $C<1/2$ due to the second example above.
In a previous version of the manuscript we incorrectly conjectured that $C$ is arbitrarily close to 1.

Theorem~\ref{thm:surfaces} follows from the following perhaps more general result. 

\begin{theorem}\label{thm:edges}
For some constant $C_2>0$, if $G$ is a multigraph with $m\ge 3$ edges and maximum multiplicity $\mu\ge1$, then $\chicon(G) \le C_2 (\mu m)^{1/4}\log (\mu m)$.
\end{theorem}

\noindent
We prove Theorems~\ref{thm:surfaces} and~\ref{thm:edges} in Section~\ref{sec:proof}.
The proof of Theorem~\ref{thm:edges} is partly probabilistic in nature. It relies on a stronger version (see Lemma~\ref{lem:edges} below) of the following simple bound. 

\begin{proposition}\label{prop:max}
If $G$ is a multigraph of maximum degree $\Delta \ge 1$, then $\chicon(G) \le \lceil\sqrt{e(2\Delta-1)}\rceil$.
\end{proposition}

\noindent
For completeness, we prove Proposition~\ref{prop:max} in Section~\ref{sec:degree} by a standard application of the Lov\'asz Local Lemma.
This has the following strong yet still partial converse, also shown in Section~\ref{sec:degree}.

\begin{proposition}\label{prop:average}
If $G$ is a multigraph of average degree $d \ge 3$, then $\chicon(G) \ge \lfloor\sqrt{d/\log d}\rfloor$.
\end{proposition}

\noindent
The last two assertions alone highlight a clear distinction between single-conflict chromatic number and, say, ordinary choosability, for which the behaviour of the complete graphs $K_{d+1}$ is linear in $d$, while that of the complete bipartite graphs $K_{d,d}$ is logarithmic in $d$~\cite{ERT80}.

Notice that Proposition~\ref{prop:average} helps to provide a broad certificate of sharpness of Theorems~\ref{thm:surfaces} and~\ref{thm:edges} up to polylogarithmic factors. This is akin to the two-vertex example exhibited earlier. In particular, consider the complete multigraph on $n$ vertices of uniform edge multiplicity $\mu$. It is a  $\mu(n-1)$-regular graph, so with $\mu \binom{n}{2}$ edges, that has $\Theta(n^2)$ Euler genus. By Proposition~\ref{prop:average}, the single-conflict chromatic number is $\Omega(\sqrt{\mu n/\log(\mu n)})$, and this is not far from the $O(\sqrt{\mu n}\log(\mu n))$ upper bound implied in both Theorems~\ref{thm:surfaces} and~\ref{thm:edges}.

It may be challenging to eliminate the logarithmic factors in Theorems~\ref{thm:surfaces} and~\ref{thm:edges}.
Since we do not know the correct asymptotics in these results, we have made no effort to optimise the values of $C_1$ and $C_2$.
On the other hand, we managed to avoid the logarithmic factors for separation and adaptable choosability (see Theorems~\ref{thm:cham} and~\ref{thm:chasurfaces} below). The simpler argument uses Proposition~\ref{prop:max} directly (rather than needing Lemma~\ref{lem:edges}), and we present it in Section~\ref{sec:proof} as a warm up to proving our main result.

One might wonder if degeneracy could be an alternative way to prove Theorem~\ref{thm:surfaces}, at least in the $\mu=1$ case. That was essentially Heawood's original approach to bounding the chromatic number.
As we will see in Section~\ref{sec:definitions}, density considerations have some use (see Lemma~\ref{prop:orientation} below); however, a construction of Kostochka and Zhu~\cite{KoZh08} for the adaptable chromatic number shows that there are graphs of degeneracy $d$ which have single-conflict chromatic number greater than $d$. There might yet be some constant $C'>0$ such that the single-conflict chromatic number of any $d$-degenerate graph on $n$ vertices is at most $C'\sqrt{d}\log n$ (which would imply the $\mu=1$ case of Theorem~\ref{thm:surfaces}), but we have not been able to prove this thus far. Theorem~\ref{thm:edges} implies an upper bound of $C_2(nd)^{1/4} \log(nd)$ in this situation.

For small $g$, it would be interesting to determine the optimal upper bound on $\chicon(G)$ over all multigraphs $G$ embeddable on a surface of Euler genus $g$ in terms of the maximum multiplicity $\mu$, particularly for the boundary cases $\mu=1$ and $\mu\to\infty$.
As we will indicate in Section~\ref{sec:definitions} it is easy to verify that the extremal single-conflict chromatic number for simple planar graphs is $4$, but we have not further investigated the precise values for $\mu=1$ with, say, $g=1,2,3$. For large $\mu$ (and fixed $g$) the second instructive example above and Theorem~\ref{thm:surfaces} together give a value asymptotically between $2\mu$ and $8\mu$, and it is tempting to narrow this range.

\subsection{Probabilistic preliminaries}\label{sub:probabilistic}

We make use of the following basic probabilistic tools. We refer the reader to the monograph of Molloy and Reed~\cite{MoRe02} for further details.

\begin{chernoff}
For any $0 \le t \le np$, \[\Pr(|\Bin(n,p)-np| > t) < 2\exp(-t^2/(3np)).\]
\end{chernoff}

\begin{slll}
Consider a set $\cal E$ of (bad) events such that for each $A\in \cal E$
\begin{enumerate}
\item $\Pr(A) \le p < 1$, and
\item $A$ is mutually independent of a set of all but at most $d$ of the other events.
\end{enumerate}
If $ep(d+1)\le1$, then with positive probability none of the events in $\cal E$ occur.
\end{slll}

\begin{glll}
Consider a set ${\cal E}=\{A_1,\dots,A_n\}$ of (bad) events such that
each $A_i$ is mutually independent of ${\cal E}-({\cal D}_i\cup A_i)$, for some ${\cal D}_i \subseteq {\cal E}$.
If we have reals $x_1,\dots,x_n \in[0,1)$ such that for each $i$
\[
\Pr(A_i) \le x_i \prod_{A_j\in {\cal D}_i} (1-x_j),
\]
then the probability that none of the events in $\cal E$ occur is at least $\prod_i (1-x_i)>0$.

\end{glll}

\section{Definitions}\label{sec:definitions}

In this section, we give some more definitions, one of single-conflict chromatic number, one of adaptable choosability, and one of separation choosability. We also show how these three parameters are related, and give a few comments related to planar graphs.

First we give an alternative definition of single-conflict chromatic number, which may be insightful.
Let $G=(V,E)$ be a multigraph.
Given a local $k$-partition $\{\kappa_v\}$ of $G$, we say $G$ is {\em conflict $\{\kappa_v\}$-orientable} if there is some orientation of all edges of $G$ such that  for every vertex $v\in V$, the set of local colours of $v$ associated to the (oriented) edges leaving $v$ does not contain all of $[k]$.
Then the single-conflict chromatic number $\chicon(G)$ of $G$ is the least $k$ such that $G$ is conflict $\{\kappa_v\}$-orientable for any local $k$-partition $\{\kappa_v\}_{v\in V}$.

\begin{proof}[Proof of equivalence]
Let $G= (V,E)$ and fix a local $k$-partition $\{\kappa_v\}$ of $G$.
It suffices to show that $G$ is conflict $\{\kappa_v\}$-orientable if and only if it is conflict $\{\kappa_v\}$-colourable. 
If it has a conflict $\{\kappa_v\}$-orientation, then for every $v\in V$ choose a colour from $[k]$ that is absent from the local colours of $v$ associated to the edges leaving $v$ to produce a conflict $\{\kappa_v\}$-colouring.
If it has a conflict $\{\kappa_v\}$-colouring $c$, then orient towards $v$ all incident edges $e$ such that $\kappa_v(e)=c(v)$ to produce a conflict $\{\kappa_v\}$-orientation.
\end{proof}

From this equivalence, the following proposition becomes plain.

\begin{proposition}\label{prop:orientation}
If there is an orientation of $G$ such that every vertex has maximum outdegree less than $k$, then $\chicon(G) \le k$.
\end{proposition}

\noindent
This implies $\chicon(G) \le 1+\max_{S\subseteq V} \lceil|E(G[S])|/|S|\rceil$, cf.~e.g.~\cite[Lem.~3.1]{AlTa92}.

\begin{corollary}\label{cor:arboricity}
If $G$ is a planar graph, then $\chicon(G) \le 4$.
If $G$ is a triangle-free planar graph, then $\chicon(G) \le 3$.
If $G$ is a a simple graph embeddable on a surface of Euler genus $g>0$, then $\chicon(G) \le H_g/2+1$, where $H_g$ is Heawood's formula for Euler genus $g$.
\end{corollary}
\noindent
Recall that every $k$-degenerate graph has an orientation of maximum outdegree at most $k$. So Proposition~\ref{prop:orientation} cannot be improved in general, since
 there are $k$-degenerate graphs with adaptable chromatic number greater than $k$~\cite{KoZh08} (and, as we will shortly see, the same then is true of single-conflict chromatic number).

\medskip
Next we discuss how single-conflict chromatic number is connected to two colouring parameters, both of which are weaker versions of 
{\em list colouring}, as introduced independently by Erd\H{o}s, Rubin and Taylor~\cite{ERT80} and by Vizing~\cite{Viz76}.

For completeness, we recall the classic definition.
Let $G = (V,E)$ be a (multi)graph. 
For a positive integer $k$, a mapping $L:V\to \binom{\mathbb Z^+}{k}$ is called a {\em $k$-list-assignment} of $G$, and a colouring $c$ of $V$ is called an {\em $L$-colouring} if $c(v)\in L(v)$ for any $v\in V$.
We say $G$ is {\em $k$-choosable} if there is a proper $L$-colouring of $G$ for any $k$-list-assignment $L$. The {\em choosability $\ch(G)$} of $G$ is the least $k$ such that $G$ is $k$-choosable.

\subsection{Adaptable choosability}\label{sub:adapt}

The following list colouring parameter was proposed by Kostochka and Zhu~\cite{KoZh08}.
Let $G=(V,E)$ be a multigraph. Given a labelling $\ell: E \to {\mathbb Z}^+$ of the edges, a (not-necessarily-proper) vertex colouring $c: V \to {\mathbb Z}^+$ is {\em adapted} to $\ell$ if for every edge $e=uv \in E$ not all of $c(u)$, $c(v)$ and $\ell(e)$ are the same value.
We say that $G$ is {\em adaptably $k$-choosable} if for any $k$-list-assignment $L$ and any labelling $\ell$ of the edges of $G$, there is an $L$-colouring of $G$ that is adapted to $\ell$. The {\em adaptable choosability} $\ch_a(G)$ of $G$ is the least $k$ such that $G$ is adaptably $k$-choosable. Every proper colouring is adapted to any labelling $\ell$, so $\ch(G)\ge \ch_a(G)$ always.

We observe adaptable choosability is at most the single-conflict chromatic number.

\begin{observation}\label{obs:dp1}
For any multigraph $G$, $\chicon(G)\ge \ch_a(G)$.
\end{observation}

\begin{proof}
Fix $G=(V,E)$ and let $k=\chicon(G)$. Let $L$ be a $k$-list-assignment and let $\ell$ be a labelling of the edges of $G$.
For each $v\in V$, locally colour each edge $e$ incident to $v$ with colour $a$ if $a\in L(v)$ and $\ell(e)=a$. This yields a local $k$-partition $\{\kappa_v\}$ (as mentioned in the introduction, it is not important that the image of each map $\kappa_v$ is equal to $[k]$, the image of of each $\kappa_v$ can be different sets of $k$ elements for each vertex $v$).
By the choice of $k$ there must be a conflict $\{\kappa_v\}$-colouring. 
It follows from our definition of $\{\kappa_v\}$ that this corresponds to an $L$-colouring that is adapted to $\ell$.
\end{proof}

We remark that adaptable choosability is in turn a strengthening of the adaptable chromatic number (for which the list assignment always takes all lists equal) and Hell and Zhu~\cite{HeZh08} have exhibited planar graphs with adaptable chromatic number at least $4$. So the single-conflict chromatic number is also exactly $4$ for such graphs.

\subsection{Separation choosability}\label{sub:separation}

The following list colouring parameter was proposed by Kratochv\'il, Tuza and Voigt~\cite{KTV98a}.
Let $G = (V,E)$ be a graph. 
We say a $k$-list-assignment $L$ has {\em maximum separation} if $|L(u)\cap L(v)| \le 1$ for every edge $uv$ of $G$.
We say $G$ is {\em separation $k$-choosable} if there is a proper $L$-colouring of $G$ for any $k$-list-assignment $L$ that has maximum separation. The {\em separation choosability $\ch_{\sep}(G)$} of $G$ is the least $k$ such that $G$ is separation $k$-choosable.
Since the choosability $\ch(G)$ of $G$ omits any separation requirement on the lists, $\ch(G) \ge \ch_{\sep}(G)$ always.

Let us see that separation choosability is at most adaptable choosability. This observation was made earlier~\cite{EKT19}, but we include it here for cohesion.
\begin{observation}\label{obs:adapt}
For any simple graph $G$, $\ch_a(G) \ge \ch_{\sep}(G)$.
\end{observation}
\begin{proof}
Fix $G=(V,E)$ and let $k=\ch_a(G)$. Let $L$ be a $k$-list-assignment of maximum separation. Let $\ell$ be a labelling defined for each $uv\in E$ by taking $\ell(uv)$ as the unique element of $L(u)\cap L(v)$ if it is nonempty, and arbitrary otherwise. By the choice of $k$, there is guaranteed to be an $L$-colouring $c$ that is adapted to $\ell$. Due to the maximum separation property of $L$ and the definition of $\ell$, the colouring $c$ must be proper.
\end{proof}

Single-conflict chromatic number is a direct strengthening of separation choosability, in the same way that ``DP-colouring'' is a strengthening of choosability~\cite{DvPo17}.

\begin{proof}[Alternative proof that $\ch_{\sep}(G) \le \chicon(G)$ for any simple graph $G$]
Fix $G=(V,E)$ and let $k=\chicon(G)$. Let $L$ be a $k$-list-assignment of maximum separation. Let $\{\kappa_v\}$ be a local $k$-partition of $G$ defined as follows. 
For each edge $e=uv\in E$, if $i$ is the unique colour in $L(u)\cap L(v)$, then let $\kappa_u(e)=i$ and  $\kappa_v(e)=i$. By the choice of $k$, there is guaranteed to be a conflict $\{\kappa_v\}$-colouring $c$. Due to the maximum separation property of $L$ and the definition of $\{\kappa_v\}$, the colouring $c$ is proper.
\end{proof}

We remark that Kratochv\'il, Tuza and Voigt~\cite{KTV98b} proved that $\ch_{\sep}(K_n) \sim \sqrt{n}$ as $n\to\infty$ by the use of affine planes. This is enough to certify sharpness of our Theorems~\ref{thm:surfaces} and~\ref{thm:edges} each up to a logarithmic factor (and Proposition~\ref{prop:max} up to a constant factor) for simple graphs.

We also note that \v{S}krekovski~\cite{Skr01} conjectured that every planar graph has separation choosability at most $3$, but this remains open to the best of our knowledge. If true, it would imply that separation choosability and adaptable choosability can be distinct for some planar graphs.

\section{Degree}\label{sec:degree}

In this section, we for completeness give the proofs of Propositions~\ref{prop:max} and~\ref{prop:average}. These results closely relate single-conflict chromatic number to the maximum and average degrees, respectively, of the multigraph.

The following proof is analogous to proofs for separation and adaptable choosability~\cite{KTV98b,KoZh08}.

\begin{proof}[Proof of Proposition~\ref{prop:max}]
Let $G =(V,E)$ be a multigraph of maximum degree $\Delta$ and fix $k= \lceil\sqrt{e(2\Delta-1)}\rceil$.
Let $\{\kappa_v\}$ be a local $k$-partition of $G$. Consider a random colouring $c:V\to [k]$ where each vertex is given an independent uniform choice. For each edge $e=uv\in E$, let $A_e$ be the event that $c(u)=\kappa_u(e)$ and $c(v)=\kappa_v(e)$. For all $e\in E$, $\Pr(A_e) = 1/k^2$ and $A_e$ is mutually independent of all but at most $2\Delta-2$ other events $A_f$. Observe that $c$ is a conflict $\{\kappa_v\}$-colouring if and only if all the events $A_e$ do not occur. The Lov\'asz Local Lemma guarantees with positive probability a conflict $\{\kappa_v\}$-colouring if $e(2\Delta-1)/k^2 < 1$, which follows from the choice of $k$.
\end{proof}

Note that the bound $\sqrt{e(2\Delta-1)}$ in Proposition~\ref{prop:max} can be slightly improved to $2\sqrt{\Delta}$ using the Local Cut Lemma~\cite[Theorem 3.1]{Ber17} instead of the Lov\'asz Local Lemma, using the same set of bad events. We have deliberately chosen to present the simpler, weaker bound.

\smallskip

The following proof is analogous to that in~\cite{KPV05} or in~\cite{Ber16}.

\begin{proof}[Proof of Proposition~\ref{prop:average}]
Let $G =(V,E)$ be a multigraph of average degree $d= 2m/n$, where $n =|V|$ and $m=|E|$. Let $k = \lfloor\sqrt{d/\log d}\rfloor$ and consider a random local $k$-partition $\{\kappa_v\}$ of $G$ where, for each edge $e=uv\in E$, the pair $(\kappa_u(e),\kappa_v(e))$ is independently, uniformly chosen from pairs in $[k]^2$. For any fixed $c:V\to[k]$,  $c$ is a conflict $\{\kappa_v\}$-colouring with probability $(1-1/k^2)^m$. By the union bound and Markov's inequality, the probability that $G$ is conflict $\{\kappa_v\}$-colourable is at most $k^n(1-1/k^2)^m \le k^n\exp(-m/k^2)$. Since $G$ has average degree $d$, we have by the choice of $k$ that $k^2\log k=d\log k/\log d< d/2=m/n$. This implies $n\log k-m/k^2<0$ and so $k^n\exp(-m/k^2)<1$. We have thus shown that with positive probability there is a local $k$-partition $\{\kappa_v\}$ for which $G$ is not conflict $\{\kappa_v\}$-colourable.
\end{proof}

We remark that since $(1+o(1))\log_2 d \le \ch_{\sep}(K_{d,d}) \le \ch_{a}(K_{d,d}) \le \ch(K_{d,d}) \le (1+o(1))\log_2 d$ as $d\to\infty$~\cite{ERT80,FKK14}, Proposition~\ref{prop:average} implies that the ratio between single-conflict chromatic number and  choosability or adaptable choosability or separation choosability can be arbitrarily large even for bipartite graphs.

\section{Proof of Theorem~\ref{thm:surfaces}}\label{sec:proof}

As a warm up to the main proof, we show the following result, an adaptable choosability analogue of Theorem~\ref{thm:edges}.

\begin{theorem}\label{thm:cham}
If $G$ is a multigraph with $m\ge 2^{16}$ edges and maximum multiplicity $\mu\ge 1$, then   $\ch_{a}(G) \le 2^{11/4}\sqrt{e} (\mu m)^{1/4}$. 
\end{theorem}
The proof of Theorem~\ref{thm:cham}  
can be viewed as a simplified version of the proof of Theorem~\ref{thm:edges}. Afterwards, we show how the following result, an adaptable choosability analogue of Theorem~\ref{thm:surfaces}, is a consequence of Theorem~\ref{thm:cham}. (At the same time, we also show how Theorem~\ref{thm:edges} implies Theorem~\ref{thm:surfaces}.)

\begin{theorem}\label{thm:chasurfaces}
For some constant $C_3>0$, if $G$ is a multigraph of maximum multiplicity $\mu\ge 1$ that is embeddable on a surface of Euler genus $g$, then $\ch_{a}(G) \le C_3 \sqrt{\mu}(g+1)^{1/4}$.
\end{theorem}

\noindent
Theorems~\ref{thm:cham} and~\ref{thm:chasurfaces} imply the same bounds for separation choosability, and both are sharp up to the choice of $C_3$ due to the complete graphs with uniform edge multiplicity $\mu$~\cite{KTV98b}.
Let us mention that the question of whether graphs of Euler genus $g$ have adaptable chromatic and choice numbers at most of order $g^{1/4}$ was first raised in December of 2007 during the {\em Graph Theory 2007} meeting in Fredericia, Denmark.

\begin{proof}[Proof of Theorem~\ref{thm:cham}]
Let $G=(V,E)$ be a multigraph with $|E|=m$ and maximum multiplicity $\mu$.
Let $k=2^{11/4}\sqrt{e}\, (\mu m)^{1/4}$,
let $L$ be a $k$-list-assignment, and consider any labelling $\ell$ of the edges of $G$. We want to prove that there is an $L$-colouring of $G$ that is adapted to $\ell$. We can assume that $G$ is connected (or else we consider each component separately), and in particular $G$ has $n\le m+1$ vertices.

Let $X=\bigcup_{v\in V} L(v)$, and let $X_1\subseteq X$ be chosen uniformly at random. Set $X_2=X\setminus X_1$. For any $i\in\{1,2\}$ and $v\in V$, $|L(v)\cap X_i|$ is binomially distributed with parameter $1/2$. The Chernoff Bound implies that $|L(v)\cap X_i|\le k/4$ with probability at most $\exp(-k/24)\le\tfrac1{2m+5}<\tfrac1{2n}$, where the first inequality uses $m\ge 2^{16}$.
By a union bound, there is a bipartition $X=X_1\cup X_2$ such that $|L(v)\cap X_i|\ge k/4$ for any $i\in\{1,2\}$ and $v\in V$.

Let $A$ be the set of vertices of degree at least $\sqrt{2\mu m}$ in $G$ and let $B=V\setminus A$. Since $|E|=m$, $A$ has most $2m/\sqrt{2\mu m} =\sqrt{2m/\mu} $ vertices, and thus $G[A]$ has maximum degree at most $\mu  \sqrt{2m/\mu}=\sqrt{2\mu m}$. By definition, $G[B]$ also has maximum degree at most $\sqrt{2\mu m}$. We remove all the colours of $X_1$ from $L(v)$ for each $v\in A$, and all the colours of $X_2$ from $L(v)$ for each $v\in B$. After this operation, each list has at least $k/4$ colours left. Since $k/4= \sqrt{2e \sqrt{2\mu m}}$, it follows from  Proposition~\ref{prop:max} that $G[A]$ has an $L$-colouring adapted to $\ell$ using only colours from $X_2$ while $G[B]$ has an $L$-colouring adapted to $\ell$ using only colours from $X_1$. Since $X_1$ and $X_2$ are disjoint, we obtain an $L$-colouring of $G$ adapted to $\ell$, as desired.
\end{proof}

Let us now see that Theorems~\ref{thm:surfaces} and~\ref{thm:chasurfaces} follow from Theorems~\ref{thm:edges} and~\ref{thm:cham}, respectively.

\begin{proof}[Proofs of Theorems~\ref{thm:surfaces} and~\ref{thm:chasurfaces}]
Assume for a contradiction that there is a counterexample $G$ to Theorem~\ref{thm:surfaces} or~\ref{thm:chasurfaces}. Take $G$ in such way that $g$ is minimised, and subject to this the number $n$ of vertices of $G$ is minimised. We can assume that $G$ is connected (or else we consider each component separately). 
Let $\tilde{G}$ be the simple graph underlying $G$.
By the minimality of $g$, $\tilde{G}$ has no embedding on a surface of smaller Euler genus, and thus has a cellular embedding on a surface $\Sigma$ of Euler genus $g$.
It follows from Euler's Formula that $\tilde{G}$ has $\tilde{m} \le 3n+3g-6$ edges, and so $G$ has $m\le \mu(3n+3g-6)$ edges.
Let $k=\max\{\lceil C_1 \sqrt{\mu}(g+1)^{1/4} \log(\mu^2(g+2))\rceil,8\mu\}$ (for Theorem~\ref{thm:surfaces}) or $k=\lceil C_3 \sqrt{\mu}(g+1)^{1/4}\rceil$ (for Theorem~\ref{thm:chasurfaces}), and assume that each vertex has $k$ local colours. If $G$ has a vertex $v$ of degree less than $k$, then remove $v$. By the minimality of $n$, we can colour $G-v$ and then find a suitable colour for $v$ (since $v$ has at least $k$ local colours and fewer than $k$ neighbours in $G$). Thus, we can assume that $G$ has minimum degree at least $k$, and thus at least $\tfrac12 nk$ edges. Consequently, $nk/2 \le \mu(3n+3g-6)$. 

For Theorem~\ref{thm:surfaces}, since $k/(2\mu) \ge 4$, we have $n\le 3g-6$ and $m\le \mu(12g-24)$. It then follows from Theorem~\ref{thm:edges} and a large enough choice of constant that $G$ has single-conflict chromatic number strictly smaller than $k$, which is a contradiction.  

For Theorem~\ref{thm:chasurfaces}, observe that not only $G$, but also $\tilde{G}$ has minimum degree at least $k$. Thus $nk/2\le \tilde{m} \le 3n+3g-6$. For a large enough choice of constant $C_3$, $k\ge 8$ and thus $n\le 3g-6$ and $m\le \mu(12g-24)$. It then follows from Theorem~\ref{thm:cham} and a large enough choice of constant that $G$ has single-conflict chromatic number strictly smaller than $k$, which is a contradiction.  
\end{proof}

To prove Theorem~\ref{thm:edges}, we require the following slightly technical result.

\begin{lemma}\label{lem:edges}
For any $d\ge 2^{23}$, let $G=(V,E)$ be a multigraph with a vertex partition $V=A\cup B$ such that
\begin{enumerate}
\item
the induced submultigraph $G[A]$ has maximum degree at most $d$,
\item
all vertices in $A$ have maximum degree at most $d^2$ in $G$, and
\item 
all vertices in $B$ have maximum degree at most $d$ in $G$.
\end{enumerate}

There is a constant $C>0$ such that for any local $k$-partition $\{\kappa_v\}$ of $G$, where $k\ge C \sqrt{d} \log d$, there is a colouring $c: A\to [k]$ such that $c$ is a conflict $\{\kappa_v\}$-colouring of $G[A]$ and no vertex $x\in B$ has more than $\sqrt{d}$ incident edges $e=xy$, $y\in A$, such that $c(y)=\kappa_y(e)$.
\end{lemma}

\begin{proof}[Proof of Theorem~\ref{thm:edges}]
Let $G=(V,E)$ be a multigraph with $m$ edges and maximum multiplicity $\mu$.
Let $A$ be the set of vertices of degree at least $\sqrt{2\mu m}$ in $G$ and let $B=V\setminus A$. Since $|E|=m$, $A$ has at most $2m/\sqrt{2\mu m} =\sqrt{2m/\mu} $ vertices, and thus $G[A]$ has maximum degree at most $\mu  \sqrt{2m/\mu}=\sqrt{2\mu m}$. It follows from the definition of $B$ that $G[B]$ also has maximum degree at most $\sqrt{2\mu m}$.
Note that $\chicon(G)$ is trivially at most $m$.
So by a large enough fixed choice of $C_2$ we may assume $m$ is large enough so that the conditions of Lemma~\ref{lem:edges} are satisfied with $d=\sqrt{2\mu m}$. Let $C>0$ be the constant associated to the corresponding application of Lemma~\ref{lem:edges}. Let $k$ be an integer at least $\max\{C \sqrt{d} \log d,\lceil\sqrt{e(2d-1)}\rceil+ \sqrt{d}\}$ and let $\{\kappa_v\}$ be a local $k$-partition of $G$. It follows from an application of Lemma~\ref{lem:edges} that there is a conflict $\{\kappa_v\}$-colouring $c$ of $G[A]$. It remains to colour $B$ in such a way that it is compatible with $c$.

For each vertex $x\in B$, remove from $G$ any edge $f$ incident to $x$ if there exists some incident edge $e=xy$, $y\in A$, such that $c(y)=\kappa_y(e)$ and $\kappa_x(f)=\kappa_x(e)$.
We also (locally) remove each of the colours associated to the edges we removed.
By one of the properties of $c$ guaranteed by Lemma~\ref{lem:edges}, this process removes at most $\sqrt{d}$ of the colours incident to each vertex in $B$. By arbitrarily deleting any excess local colours as well as any of the incident edges with those colours, then relabelling colours, we are left with a local $k'$-partition $\{\kappa_v'\}$ of a submultigraph of $G[B]$ with maximum degree at most $d$, where $k'=\lceil\sqrt{e(2d-1)}\rceil$. By Proposition~\ref{prop:max}, this submultigraph admits a conflict $\{\kappa_v'\}$-colouring $c'$. The colour and edge removal process we performed ensures that, by reversing the relabelling, $c'$ corresponds to a conflict $\{\kappa_v\}$-colouring of $G[B]$ that combines with $c$ to produce a conflict $\{\kappa_v\}$-colouring of all of $G$.
\end{proof}

It remains only to prove Lemma~\ref{lem:edges}. This is done with an application of the General Local Lemma (see Section~\ref{sub:probabilistic}).

\begin{proof}[Proof of Lemma~\ref{lem:edges}]
Let $k= \lceil C \sqrt{d} \log d \rceil$ where $C$ is some constant large enough to guarantee certain properties as specified later in the proof. Let $\{\kappa_v\}$ be a local $k$-partition of $G$.

We must do a pruning operation before proceeding ---in fact, this is the crucial step in the proof.
By taking $C$ large enough, we may assume for each $v\in A$ and each $i\in [k]$ that the number of edges in $\kappa_v^{-1}(i)$ with its other endpoint also in $A$ is at most $\sqrt{d}$.
(We summarily remove all edges associated to every colour not satisfying the property, and since the maximum degree of $G[A]$ is at most $d$ this removes at most $\sqrt{d}$ of the colours around each vertex in $A$.)

Let $p = 2^{-4}/\sqrt{d}$. Consider a random selection of colours where each of the $|V|k$ local colours is selected according to an independent Bernoulli trial of probability $p$. With an eye to applying the General Local Lemma, let us define three types of (bad) events.
\begin{enumerate}
\renewcommand\labelenumi{\Roman{enumi}}
\item
For a vertex $x\in A$, none of the colours around $x$ is selected.
\item
For an edge $e=xy\in E$ with $x,y\in A$, $\kappa_x(e)$ and $\kappa_y(e)$ are both selected.
\item
For a vertex $x\in B$, there are more than $\sqrt{d}$ edges $e=xy$, $y\in A$, for which $\kappa_y(e)$ is selected.
\end{enumerate}
If we obtain a selection for which none of the above events occurs, then we are done. This is because the deselection of a colour does not introduce any new event of Type~II or~III. So we can arbitrarily deselect all but one of the colours around each vertex, and the remaining selection induces the desired colouring $c$, thanks to the fact that no events of Type~II or~III hold.

For each $x\in A$, the probability of a Type~I event is $\Pr(\Bin(k,p) = 0) = (1-p)^k\le \exp(-pk)\le \exp(-2^{-4}C\log d) <2^{-8}/d$ if $C$ is chosen large enough.
For each edge $e=xy\in E$, $x,y\in A$, the probability of a Type~II event is $p^2=2^{-8}/d$.
For each vertex $x\in B$, the probability of a Type~III event is at most $\Pr(\Bin(d,p) > \sqrt{d}) \le \Pr(|\Bin(d,p)-dp| > \sqrt{3}\cdot 2^{-4}\sqrt{d}) < 2\exp(-2^{-4}\sqrt{d})$ by the Chernoff Bound.

The choice to generate the random colouring according to independent Bernoulli trials rather than a uniform colour per vertex (as in Proposition~\ref{prop:max}) is important for us in establishing the following bounds on dependence between bad events, especially for Type~III events.
Each Type~I event is mutually independent of all but at most $d$ events of Type~I, at most $d$ events of Type~II, and at most $d^2$ events of Type~III.
Each Type~II event is mutually independent of all but at most $2$ events of Type~I, at most $2d-1$ events of Type~II, and at most $2d^2$ events of Type~III.
By the pruning operation we did at the beginning, each Type~III event is mutually independent of all but at most $d$ events of Type~I, at most $d^{3/2}$ events of Type~II, and at most $d^{3/2}$ events of Type~III.
(To be more explicit, each Type~III event is determined by up to $d$ independent Bernoulli random variables, each of which corresponds to a local colour of a neighbour. Thanks to the pruning, the number of Type~II events, say, that also use this randomness is at most $d^{3/2}$. The Type~III event is mutually independent of all other Type~II events.)

We associate weight $x_i = 2^{-7}/d$ to each event $i$ of Type~I or~II, and weight $x_i = 2\exp(-2^{-6}\sqrt{d})$ to each event $i$ of Type~III.
By the considerations above, the General Local Lemma guarantees the desired selection of colours with positive probability, provided the following three inequalities hold (where we repeatedly used that $\exp(-x-x^2)\le 1-x \le \exp(-x)$  if $0<x<0.69$):
\begin{align*}
& 1/2 \le \exp\left(-\frac{1}{2^{7}} -\frac{1}{2^{14}d} -\frac{1}{2^{7}} -\frac{1}{2^{14}d} -\frac{2d^2}{\exp(\frac{\sqrt{d}}{2^{6}})} -\frac{4d^2}{\exp(\frac{\sqrt{d}}{2^{5}})}\right); & \rm(I)\\
%
& 1/2 \le \left(1-\frac{1}{2^{7}d}\right)^2\exp\left(-\frac{1}{2^{6}} -\frac{1}{2^{13}d} -\frac{4d^2}{\exp(\frac{\sqrt{d}}{2^{6}})} -\frac{8d^2}{\exp(\frac{\sqrt{d}}{2^{5}})}\right); & \rm(II)\\
%
& -\frac{\sqrt{d}}{2^4} +\frac{\sqrt{d}}{2^6} \le -\frac{1}{2^{7}} -\frac{1}{2^{14}d} -\frac{\sqrt{d}}{2^{7}} -\frac{1}{2^{14}\sqrt{d}} -\frac{2d^{3/2}}{\exp(\frac{\sqrt{d}}{2^6})}  -\frac{4d^{3/2}}{\exp(\frac{\sqrt{d}}{2^{5}})}. & \rm(III)
\end{align*}
It is straightforward to check that $d\ge 2^{23}$ suffices.
\end{proof}

The above proof can be straightforwardly adapted for the same upper bound (with a larger constant $C$) on a stronger type of single-conflict chromatic number where additionally we must assign $\Omega(\log d)$ distinct colours per vertex instead of just one.
What this then directly implies is that, for any simple graph $G$ that is embeddable on a surface of Euler genus $g$, the single-conflict chromatic number is $O(g^{1/4} (\log g)^{5/4})$ even if we allow $O(\log g)$ conflicts per edge and demand $\Omega(\log d)$ distinct colours per vertex.

\subsection*{Notes added}

In a version of this work first circulated on arXiv (\href{https://arxiv.org/abs/1803.10962v1}{arXiv:1803.10962v1}), we called $\chicon$ the {\em least conflict choosability} and denoted it instead by $\ch_{\con}$. Upon the suggestion of a referee, we have reformulated our terminology to better place it amongst the extant colouring notions.

It transpires that single-conflict chromatic number is also naturally related to the classic problem of finding independent transversals in vertex-partitioned graphs. Specifically, given $G$ with a local $k$-partition $\{\kappa_v\}$, one can define its {\em cover graph} $H$ and an associated vertex partition as follows. The vertex set $V(H)$ consists of all pairs $(v,i)$ with $v\in V(G)$ and $i\in[k]$. The edge set $E(H)$ includes $(v,i)(v',i')$ if there is an edge $e=vv'\in E(G)$ such that $\kappa_v(e)=i$ and $\kappa_{v'}(e)=i'$. The parts of $H$ are defined according to $V(G)$, i.e.~for each $v\in V(G)$ the vertices $(v,i)$, $i\in[k]$, are all in one part.
Then $G$ is conflict $\{\kappa_v\}$-colourable if and only if $H$ contains an independent set that is transversal to the partition of $H$.
Thus one may convert between results on single-conflict chromatic number and on independent transversals.
For instance, one may recast Proposition~\ref{prop:max} as a result about independent transversals subject to some average degree condition. See~\cite{KaKe20+}
for more discussion on this perspective and related references.

\subsection*{Acknowledgement}
We are grateful to the anonymous referees for helpful suggestions leading to improvements in the presentation of our work.

\bibliographystyle{abbrv}
\bibliography{dp1}

\end{document}